\documentclass[a4paper,11pt]{article}

\usepackage{amsfonts,amssymb,amsmath,amsthm,latexsym,epsf,epsfig,amscd,bbm,stmaryrd,stackrel}
\usepackage[english]{babel}

\makeatletter \@addtoreset{equation}{section} \makeatother
\makeatletter \@addtoreset{enunciato}{section} \makeatother
\newcounter{enunciato}[section]

\newtheorem{ittheorem}{Theorem}
\newtheorem{itlemma}{Lemma}
\newtheorem{itproposition}{Proposition}
\newtheorem{itdefinition}{Definition}
\newtheorem{itremark}{Remark}
\newtheorem{itclaim}{Claim}
\newtheorem{itfact}{Fact}
\newtheorem{itconjecture}{Conjecture}
\newtheorem{itcorollary}{Corollary}

\newcommand{\bb}[1]{{\mathbb #1}}

\newenvironment{theorem}{\addtocounter{enunciato}{1}
\begin{ittheorem}}{\end{ittheorem}}
\newenvironment{lemma}{\addtocounter{enunciato}{1}
\begin{itlemma}}{\end{itlemma}}
\newenvironment{proposition}{\addtocounter{enunciato}{1}
\begin{itproposition}}{\end{itproposition}}
\newenvironment{definition}{\addtocounter{enunciato}{1}
\begin{itdefinition}}{\end{itdefinition}}
\newenvironment{remark}{\addtocounter{enunciato}{1}
\begin{itremark}}{\end{itremark}}

\newenvironment{conjecture}{\addtocounter{enunciato}{1}
\begin{itconjecture}}{\end{itconjecture}}
\newenvironment{corollary}{\addtocounter{enunciato}{1}
\begin{itcorollary}}{\end{itcorollary}}
\newcommand{\be}[1]{\begin{equation}\label{#1}}
\newcommand{\ee}{\end{equation}}
\newcommand{\bl}[1]{\begin{lemma}\label{#1}}
\newcommand{\el}{\end{lemma}}
\newcommand{\br}[1]{\begin{remark}\label{#1}}
\newcommand{\er}{\end{remark}}
\newcommand{\bt}[1]{\begin{theorem}\label{#1}}
\newcommand{\et}{\end{theorem}}
\newcommand{\bd}[1]{\begin{definition}\label{#1}}
\newcommand{\ed}{\end{definition}}
\newcommand{\bp}[1]{\begin{proposition}\label{#1}}
\newcommand{\ep}{\end{proposition}}
\newcommand{\bc}[1]{\begin{corollary}\label{#1}}
\newcommand{\ec}{\end{corollary}}
\newcommand{\bcj}[1]{\begin{conjecture}\label{#1}}
\newcommand{\ecj}{\end{conjecture}}

\def \Z {{\mathbb Z}}
\def \R {{\mathbb R}}

\def \N {{\mathbb N}}

\def \ba {\begin{array}}
\def \ea {\end{array}}

\def \E {{\mathbb E}}

\def\1{\mathbbm{1}}

\begin{document}

\title{Symmetric exclusion \\as a model of non-elliptic\\
 dynamical random conductances}

\author{\renewcommand{\thefootnote}{\arabic{footnote}}
L.\ Avena \footnotemark[1]}

\footnotetext[1]{Institut f\"ur Mathematik, Universit\"at
Z\"urich, Winterthurerstrasse 190, Z\"urich, CH- 8057,
Switzerland. \\ E-mail: luca.avena@math.uzh.ch.}

\maketitle

\begin{abstract}
We consider a finite range symmetric exclusion process on the
integer lattice in any dimension. We interpret it as a
non-elliptic time-dependent random conductance model by setting
conductances equal to one over the edges with end points
occupied by particles of the exclusion process and to zero elsewhere. We
prove a law of large number and a central limit theorem for the
random walk driven by such a dynamical field of conductances by
using the Kipnis-Varhadan martingale approximation. Unlike the
tagged particle in the exclusion process, which is in some sense
similar to this model, this random walk is diffusive even in the
one-dimensional nearest-neighbor symmetric case.

\vspace{0.5cm}\noindent

\end{abstract}
\newpage


\section{Introduction}
\subsection{Model and results}
\label{s1.1}
Let $\Omega = \{0,1\}^{\bb Z^d}$. Denote by $\xi = \{\xi(z); z \in \bb Z^d\}$
the elements of $\Omega$.
For $\xi \in \Omega$ and $y,z \in \bb Z^d$, define $\xi^{y,z} \in \Omega$ as
\[\xi^{y,z}(x) =
\left\{
\begin{array}{cl}
\xi(z), & x=y\\
\xi(y), & x=z\\
\xi(x), & x \neq z,y,
\end{array}
\right.
\]
that is, $\xi^{y,z}$ is obtained from $\xi$ by exchanging the
occupation variables at $y$ and $z$. Fix $R\geq1$. Consider the
transition kernel $p(z,y)$ of a translation-invariant, symmetric,
irreducible random walk with range size $R$, i.\ e.,
$p(0,y-z)=p(z,y)=p(y,z)>0$ iff $|z-y|_1\leq R$, and
$\sum_{y\in\Z^d}p(0,y)=1$. Due to translation invariance we will
denote $p(x):=p(0,x)$.

Let $\{(\xi_t,X_t);t \geq 0\}$ be the Markov process on the state
space $\Omega \times \bb Z^d$ with generator given by
\begin{equation}
\begin{aligned}
L f(\xi,x) &= \sum_{y,z\in\,\bb Z^d} p(z-y)\big[f(\xi^{y,z},x)-f(\xi,x)\big]\\
&+\sum_{y\in\,\bb Z^d}c_{x,y}(\xi)\big[f(\xi,y)-f(\xi,x)\big],
\end{aligned}
\end{equation}
for any local function $f: \Omega \times \bb Z^d \to \bb R$, with
\be{RC}
c_{x,y}(\xi)=
\left\{\begin{array}{ll}
\xi(x)\xi(y) &\mbox{ if } |x-y|_1\leq R ,\\
0 &\mbox{else}.
\end{array}
\right.
\ee

We interpret the dynamics of the process $\{(\xi_t,X_t);t \geq
0\}$ as follows. Checking the action of $L$ over functions $f$
which do not depend on $z$, we see that $\{\xi_t; t \geq 0\}$ has
a Markovian evolution, which corresponds to the well known \emph{
symmetric exclusion process} on $\bb Z^d$, see e.g. \cite{Li85}.
Conditioned on a realization of $\{\xi_t; t \geq 0\}$, the process
$\{X_t; t \geq
0\}$ is a continuous time random walk among the field of dynamical random
conductances
\begin{equation}\{c_{x,y}(\xi_t)=\xi_t(x)\xi_t(y)\1_{\{|x-y|_1\leq R\}} |\, x,y\in\Z^d, t\geq 0\}
.\end{equation}

Our main results are the following law of large numbers and
functional central limit theorem for the random walk $X_t$.
\begin{theorem}{\bf{(LLN)}}\label{LLN}
Assume that the exclusion process $\xi_t$ starts from the Bernoulli product measure $\nu_{\rho}$ 
of density
$\rho\in[0,1]$. Then $X_t/t$ converges a.s. and in $L_1$
to $0$.\end{theorem}
\begin{theorem}{\bf{(Annealed functional CLT)}}\label{Functional CLT}
Under the assumptions of Theorem \ref{LLN}, the process
$(\epsilon X_{t/\epsilon^2})$ converges in distribution, as
$\epsilon$ goes to zero, to a non-degenerate Brownian motion
with covariance $\sigma$ in the Skorohod topology.\end{theorem}

\subsection{Motivation}
Random walks in random media represents one of the main research
area within the field of disordered system of particles. The aim
is to understand the motion of a particle in a inhomogeneous
medium. This is clearly interesting for applied purposes and has
turned out to be a very challenging mathematical program. Lots of
effort has been made in recent years in this direction. We refer
to \cite{Sz02,Ze06} for recent overviews of rigorous results in
this subject.

One of the easiest models of random walk in random media is
represented by a random walk among (time-independent) random
conductances. This model turned out to be relatively simple due
the reversibility properties of the walker. In fact, the behavior
of such random walks has been recently analyzed and understood in
quite great generality. See, e.\ g.\ ,\cite{Bi11} for a recent overview
and references therein.
When considering a field of dynamical random conductances, the
mentioned reversibility of the random walk is lost, and other types
of techniques are needed. In the recent paper \cite{An12}, 
annealed and quenched invariance principles for
a random walk in a field of time-dependent random
conductances have been derived by assuming fast enough space-time
mixing conditions and uniform ellipticity for the field. In
particular, the uniform ellipticity, which guarantees heat kernel
estimates, is a crucial assumption in their approach even for the annealed 
statement(ellipticity plays a fundamental role also in the analysis 
of other random walks in random environments). 
The model we consider represents a first solvable example of non-elliptic 
time-dependent random conductances. Moreover it strengthens the connection between
particle systems theory and the theory of random walks in random media. To
overcome the loss of ellipticity we use the ``good'' properties of
the symmetric exclusion in equilibrium.

The proof of our results rely on the martingale approximation
method developed by Kipnis and Varhadan \cite{KiVa86} for additive
functionals of reversible Markov processes. In the original paper
\cite{KiVa86}, the authors apply their method to study a tagged
particle in the exclusion process. Indeed, this latter has some
similarities with our model, and our proof is essentially an
adaptation of their proof. Unlike the tagged particle behavior,
our random walk is always diffusive even in the one dimensional
nearest-neighbors case.

\section{Proofs of the LLN and of the invariance principle}

\subsection{The environment from the position of the walker}
\label{1.2} Consider the process $\{\eta_t; t \geq 0\}$ with
values in $\Omega$, defined by $\eta_t= \tau_{X_t}\xi_t$, where 
$\tau_y$ denotes the shift operator on $\Omega$ (i.\ e.\
$\eta_t(z)= \xi_t(z+X_t)$). The process $\{\eta_t; t \geq 0\}$ is
usually called the {\em environment seen by the random walk}. For
$\eta \in \Omega$, the process $\{\eta_t; t \geq 0\}$ is also
Markovian with generator:
\begin{equation}
\begin{aligned}\label{generators}
\mathcal{L}_{ew}f(\eta) &= \sum_{z,y}
p(z-y)\big[f(\eta^{y,z})-f(\eta)\big] +
\sum_{y}c_{0,y}(\eta)\big[f(\tau_y\eta)-f(\eta)\big]\\
&=:\mathcal{L}_{se}f(\eta)+\mathcal{L}_{rc}f(\eta).
\end{aligned}
\end{equation}
for any local function $f: \Omega \to \bb R$. The choice of the
subindexes in the generators above is just for notational
convenience: ``$ew$'', ``$se$'' and ``$rc$'', stand for, Environment from the point
of view of the Walker, Symmetric Exclusion and Random
Conductances, respectively.

For any function $f,g: \Omega \to \bb R$, we denote the inner
product in $L_2=L_2(\nu_\rho)$ by
$$\langle f,g\rangle_{\nu_\rho}:=\int_\Omega d\nu_\rho f(\eta)g(\eta),$$
where $\nu_{\rho}$ is the Bernoulli product measure of density
$\rho\in[0,1]$. In particular, it is well known that the family
$\{\nu_{\rho}: \rho\in(0,1)\}$ fully characterizes the set of
extremal invariant measures for the symmetric exclusion process,
and $\mathcal{L}_{se}$ is self-adjoint in $L_2$ (see e.g. \cite{Li85}).
The next lemma shows that the same hold
for the environment as seen by the walker.
Before proving it, we define the
Dirichlet forms associated to the generators involved in
\eqref{generators} as
\begin{equation}
D_a(f):=\langle f, -\mathcal{L}_{a}f\rangle_{\nu_\rho} \text{ with
} a\in\{ew,se,rc\},
\end{equation}
for $f$ in $L_2$.
It follows by a standard computation (cf. \cite{KiLa99}, Prop. 10.1 P.343) that

\begin{equation}
\begin{aligned}\label{DirichletProp}
D_{ew}(f)=&D_{se}(f)+D_{rc}(f)= \frac{1}{2}
\sum_{z,y}
\int d\nu_\rho\, p(z-y)\big[f(\eta^{y,z})-f(\eta)\big]^2
\\&+\frac{1}{2}\sum_{y}\int d\nu_\rho\, c_{0,y}(\eta)\big[f(\tau_y\eta)-f(\eta)\big]^2.
\end{aligned}
\end{equation}

\begin{lemma}\label{ergodicityEP}
The process $\eta_t$ is reversible and ergodic with respect to the
the Bernoulli product measure $\nu_\rho$.
\begin{proof}
We first show that $\mathcal{L}_{ew}$ is self-adjoint in $L_2$, namely
,$\langle f,\mathcal{L}_{ew}g\rangle_{\nu_\rho}=
\langle \mathcal{L}_{ew}f,g\rangle_{\nu_\rho}$,
with $f,g$ arbitrary functions.

By translation invariance, we have
\begin{equation}\begin{aligned}
\langle f,\mathcal{L}_{rc}g\rangle_{\nu_\rho}&= \sum_y\int d\nu_\rho\, f(\eta)
\left[g(\tau_y\eta)-g(\eta)\right]c_{0,y}(\eta)\\
&=\sum_y\left(\int d\nu_\rho\, f(\tau_{-y}\eta)g(\eta)
c_{0,-y}(\eta)-\int d\nu_\rho\, f(\eta)g(\eta)c_{0,y}(\eta)\right)\\
&=\sum_y\left(\int d\nu_\rho \, f(\tau_{y}\eta)g(\eta)c_{0,y}(\eta)
-\int d\nu_\rho \, f(\eta)g(\eta)c_{0,y}(\eta)\right)=\langle \mathcal{L}_{rc}f,g\rangle_{\nu_\rho}.
\end{aligned}
\end{equation}
Together with the fact that $\mathcal{L}_{se}$ is also
self-adjoint, we get
\begin{equation*}\langle f,\mathcal{L}_{ew}g\rangle_{\nu_\rho}=\langle f,\mathcal{L}_{se}g\rangle
_{\nu_\rho}+\langle f,\mathcal{L}_{rc}g\rangle_{\nu_\rho}=\langle
\mathcal{L}_{se}f,g\rangle_{\nu_\rho} +\langle
\mathcal{L}_{rc}f,g\rangle_{\nu_\rho}=\langle \mathcal{L}_{ew}f,g\rangle_{\nu_\rho}.
\end{equation*}

It remains to show the ergodicity. Following the argument in \cite{KiVa86},
 we show that any harmonic function $h$ such that
$\mathcal{L}_{ew}h=0$ is $\nu_\rho$-a.\ s.\ constant.

Indeed  $\mathcal{L}_{ew}h=0$ implies that $D_{se}(h)=-D_{rc}(h)$.
Since the Dirichlet forms are non-negative, then
$D_{se}(h)=0=D_{rc}(h)$, but $\mathcal{L}_{se}$ is reversible and ergodic,
hence $h$ must be $\nu_\rho$-a.\ s.\ constant.
\end{proof}
\end{lemma}

\subsection{Proof of Theorem \ref{LLN}}
We now express the position of the RW $X_t$ in terms of the
process $\eta_t$. For $y\in\Z^d$, let $J_t^y$ denote the number of
spatial shifts in direction $y$ of the process $\eta_t$ up to time
$t$. Then
\begin{equation}\label{RWfromShift}
 X_t=\sum_y yJ_t^y.
\end{equation}

By compensating the process $J_t^y$ by its intensity $\int_0^t
c_{0,y}(\eta_s) d s$, it is standard to check that
\begin{equation}\label{CompoundPP}
M_t^y:=J_t^y-\int_0^t d s \,c_{0,y}(\eta_s) \quad\text{and}\quad (M_t^y)^2-\int_0^t d s \,c_{0,y}(\eta_s)
\end{equation}
are martingales with stationary increments vanishing at $t=0$.

Next, define
\begin{equation}\label{functionalDef}M_t:=\sum_y yM_t^y\quad\text{ and }\quad
\phi(\eta_s):= \sum_y yc_{0,y}(\eta_s),
\end{equation}

by combining \eqref{RWfromShift} and \eqref{CompoundPP}, we obtain
\begin{equation}\label{RWrepresentation}
 X_t=M_t+\int_0^t d s \,\phi(\eta_s),
\end{equation}
from which we easily obtain the law of large numbers in Theorem
\ref{LLN}.
Indeed, due to Lemma \ref{ergodicityEP}, the representation in
\eqref{RWrepresentation} express $X_t$ as a sum of a zero-mean
martingale with stationary and ergodic increments $M_t$, plus the
term $\int_0^t d s \,\phi(\eta_s)$, which by the ergodic theorem,
when divided by $t$, it converges to its average
$$\E_{\nu_\rho}[\phi(\eta)]=\sum_{|y|_1\leq R}y
\int d\nu_\rho\, \eta(0)\eta(y)=\rho^2\sum_{|y|_1\leq R}y=0.$$
\subsection{Proof of Theorem \ref{Functional CLT}}
Next, we want to prove a functional CLT for the process
$X_t$. To this aim we will use again the representation
in \eqref{RWrepresentation} and the well known Kipnis-Varhadan
method \cite{KiVa86} for additive functional of reversible Markov
processes. Indeed, $\int_0^t d s \,\phi(\eta_s)$ in
\eqref{RWrepresentation} is an additive functional of the
reversible process $\eta_t$. To recall briefly the Kipnis-Varadhan
method, we first introduce the Sobolev spaces $\mathcal{H}_1$ and
$\mathcal{H}_{-1}$ associated to a generator $\mathcal{L}$. Let
$\mathcal{D}(\mathcal{L})$ be the domain of this generator.
Consider in $\mathcal{D}(\mathcal{L})$, the equivalence relation
$\mathtt{\sim}_1$ defined as $f\mathtt{\sim}_1g$ if $\|f-g\|_1=0$,
where $\|\cdot\|_1$ is the semi-norm given by
\begin{equation}\label{1-norm}
\|f\|_1^2:=\langle f, -\mathcal{L}f\rangle_{\nu_\rho}.
\end{equation}
Define the space $\mathcal{H}_1$ as the completion of the normed
space $(\mathcal{D}(\mathcal{L})|_{\mathtt{\sim}_1},\|\cdot\|_1)$.
It can be check that $\mathcal{H}_1$ is a Hilbert space with inner
product $\langle f,g\rangle_1:=\langle
f,-\mathcal{L}g\rangle_{\nu_{\rho}}$. Next, for $f\in L_2$, let
\begin{equation}\label{-1-norm}
\|f\|_{-1}:=\sup \left\{\frac{\langle
f,g\rangle_{\nu_{\rho}}}{\|g\|_1}: g\in L_2, \|g\|_1\neq 0 \right\}.
\end{equation}
Consider $\mathcal{G}_{-1}:=\{f\in L_2 : \|f\|_{-1}<\infty\}$. As
for the $\|\cdot\|_1$ norm, define the equivalence relation
$\mathtt{\sim}_{-1}$, and let $\mathcal{H}_{-1}$ be the completion
of the normed space
$(\mathcal{G}_{-1}|_{\mathtt{\sim}_1},\|\cdot\|_1)$.
$\mathcal{H}_{-1}$ is the dual of $\mathcal{H}_{1}$ and it is also
a Hilbert space.
Theorem 1.8 in \cite{KiVa86} states that, if $\mathcal{L}$ is
self-adjoint and $\phi\in \mathcal{H}_{-1}$ (which we prove in the
next lemma), then there exists a square integrable martingale
$\tilde{M}_t$ and an error term $E_t$ such that
\begin{equation}\label{MartingaleDec}
\int_0^t d s \,\phi(\eta_s)= \tilde{M}_t+E_t,
\end{equation}
and $|E_t|/\sqrt{t}$ converges to zero in $L_2$.

In particular, denoting by $\cdot$ the
standard inner product and considering a vector $l$ in $\R^d$,
the martingale $\tilde{M_t}\cdot l$ in \eqref{MartingaleDec} is obtained as the limit
as $\lambda\rightarrow 0$ of the martingale
\begin{equation}\label{lambdaMartingale}\tilde{M}_t(\lambda,l):= f_\lambda(\eta_t)-f_\lambda(\eta_0)-
\int_0^t ds \,\mathcal{L}f_\lambda(\eta_s),
\end{equation}
where $f_\lambda$ is the solution of the resolvent equation
\begin{equation}\label{resolvent}
 (\lambda I-\mathcal{L})f_\lambda= \phi\cdot l.
\end{equation}
Moreover
\begin{equation}\label{M1lambda}
\E_{\nu_{\rho}}[\tilde{M}_1(\lambda,l)^2]=\|f_{\lambda}\|_1^2.
\end{equation}

We are now ready to show the crucial estimate which, in view of
what we said, by Theorem 1.8 in \cite{KiVa86}, implies the
decomposition in \eqref{MartingaleDec}.

\begin{lemma}
There exists a constant $K>0$ such that, for any function
$f\in\mathcal{D}(\mathcal{L}_{ew})$ and $l$ in $\R^d$,
\begin{equation}\label{Dirichlet1}|\langle \phi\cdot l,f\rangle_{\nu_\rho}|\leq K\, D_{ew}
(f)^{1/2}.\end{equation}
\begin{proof}

Recall \eqref{functionalDef} and estimate
\begin{equation}\begin{aligned} \left|\langle \phi \cdot l,f\rangle_{\nu_\rho}\right|&= \left|\int d\nu_\rho \sum_y (y\cdot l)
c_{0,y}(\eta)f(\eta)\right|=\left|\frac{1}{2}\int d\nu_\rho \sum_y (y\cdot l) \left[c_{0,y}(\eta)
-c_{0,-y}(\eta)\right]f(\eta)\right|\\
&=\left|\frac{1}{2}\int d\nu_\rho \sum_y (y\cdot l) c_{0,y}(\eta)\left[f(\tau_{y}\eta)
-f(\eta)\right]\right|\\&\leq \frac{1}{2}\left(\sum_y (y\cdot l)^2
c_{0,y}(\eta)\right)^{1/2}\left(
\int d\nu_\rho \sum_y c_{0,y}(\eta)\left[f(\tau_{y}\eta)-f(\eta)\right]^2\right)^{1/2}\\
&\leq K\, D_{rc}(f)^{1/2}\leq K\, D_{ew}(f)^{1/2},
\end{aligned}\end{equation}
where we have used translation invariance,
$c_{0,y}(\eta)^2=c_{0,y}(\eta)$, Cauchy-Schwartz, the finite range
assumption on $p(\cdot)$, and the representation of the Dirichlet
forms in \eqref{DirichletProp}, respectively.
\end{proof}
\end{lemma}
In view of the discussion above, from \eqref{RWrepresentation} and \eqref{MartingaleDec}, we have that
\begin{equation}\label{RWrepresentation2}
X_t=M_t+\tilde{M}_t+o(\sqrt{t}).
\end{equation}
Since the sum of two martingales is again a martingale, the
functional CLT for $X_t$ follows immediately from the standard
functional CLT for martingales provided that we prove the
non-degeneracy of the covariance matrix of the martingale given by
$M_t+\tilde{M}_t$. Roughly speaking, we have to prove that $M_t$ and
$\tilde{M}_t$ do not cancel each other. This is the content of the
next proposition which concludes the proof of Theorem
\ref{Functional CLT}.

\begin{proposition} The sum of the two martingales $M_t+\tilde{M}_t$ is a
non-degenerate martingale.
\end{proposition}

\begin{proof}
For $z,y\in\Z^d$ with $p(z-y)>0$, let $I_t^{y,z}$ denote the total number of
jumps of particles from $y$ to $x$ up to time
$t$. Similarly to \eqref{CompoundPP}, by compensating the process $I_t^{y,z}$ by its intensity,
it is standard to check that
\begin{equation}\label{CompoundPP2}
N^{y,z}_t:=I_t^{y,z}-p(z-y)t\quad\text{and}\end{equation}
\begin{equation}\label{CompoundPP3}(N^{y,z}_t)^2-p(z-y)t
\end{equation} are martingales.

In particular, the martingales $\{M^{y}_t |\,
y \in\Z^d\}$ (recall  \eqref{CompoundPP}) and $\{N^{y,z}_t |\,
y,z\in\Z^d, p(z-y)>0\}$ are jump processes which do not have common jumps. Therefore they are
orthogonal, namely, the product of
two such martingales is still a martingale.

On the other hand, we can check that the martingale in \eqref{lambdaMartingale} can be expressed as
\begin{equation}\label{lambdaMartingale2}
\begin{aligned}
\tilde{M}_t(\lambda,l)&= \sum_{y,z} \int_0^t d N^{y,z}_s \left[f_\lambda(\eta_s^{y,z})
-f_\lambda(\eta_s)\right]\\
&+\sum_y \int_0^t d M^{y}_s \left[f_\lambda(\tau_{y}\eta_s)-f_\lambda(\eta_s)\right].
\end{aligned}
\end{equation}

Since $M_t, \tilde{M}_t$ are mean-zero square integrable martingales with stationary increments, to prove that
$M_t+\tilde{M}_t$ is a non-degenerate martingale, we show that for any vector $l\in\R^d$,
\begin{equation}\label{NonDeg}\E_{\nu_{\rho}}\left[(M_1\cdot l + \tilde{M}_1\cdot l)^2\right]>0.
\end{equation}

By using \eqref{lambdaMartingale2}, the orthogonality and the form
of the quadratic variations of $M^{y}_t$ and $N^{y,z}_t$(see
\eqref{CompoundPP} and \eqref{CompoundPP3}), and
\eqref{DirichletProp}, we have that
\begin{equation}\begin{aligned}\E_{\nu_{\rho}}\left[(M_1\cdot l + \tilde{M}_1\cdot l)^2\right]&=
\lim_{\lambda\rightarrow 0}\E_{\nu_{\rho}}\left[\left(M_1\cdot l + \tilde{M}_1(\lambda,l)\right)^2\right]
\\
&=\lim_{\lambda\rightarrow 0}\E_{\nu_{\rho}}\left[
\left(\int_0^1 \sum_{y,z}
\big[f_{\lambda}(\eta_s^{y,z})-f_{\lambda}(\eta_s)\big]d N^{y,z}_s \right)^2\right]\\
&+\lim_{\lambda\rightarrow 0}\E_{\nu_{\rho}}\left[ \left(\int_0^1
\sum_{y} \big\{(y\cdot
l)+\left[f_{\lambda}(\tau_y\eta_s)-f_{\lambda}(\eta_s)\right]
\big\}d M^{y}_s \right)^2\right]
\\
&=\lim_{\lambda\rightarrow
0}2D_{se}(f_{\lambda})
\\&+\lim_{\lambda\rightarrow 0}\E_{\nu_{\rho}}\left[
\sum_{y} c_{0,y}(\eta)
\big\{(y\cdot l)+\left[f_{\lambda}(\tau_y\eta)-f_{\lambda}(\eta)\right]\big\}^2\right].
\end{aligned}
\end{equation}

Hence, to conclude \eqref{NonDeg}, we argue a follows. Assume that
there exists a constant $K>0$ such that

\begin{equation}\label{Dirichlet2}|\langle \phi\cdot l,f_{\lambda}\rangle_{\nu_\rho}|\leq K\,
D_{se}(f_{\lambda})^{1/2},\end{equation}

then
\begin{equation}\label{sandwich}
D_{ew}(f_{\lambda})\leq |\langle
\phi\cdot l,f_{\lambda}\rangle_{\nu_\rho}|\leq K\,
D_{se}(f_{\lambda})^{1/2},
\end{equation}
where the first inequality follows by $D_{ew}(f_{\lambda})\leq
D_{ew}(f_{\lambda})+\lambda|\langle
f_{\lambda},f_{\lambda}\rangle_{\nu_\rho}|=|\langle
\phi\cdot l,f_{\lambda}\rangle_{\nu_\rho}|$.

In view of \eqref{sandwich}, if $D_{ew}(f_{\lambda})$ stays
positive in the limit as $\lambda\rightarrow 0$, the same holds
for $D_{se}(f_{\lambda})$ and the variance is positive. On the
other hand, if $D_{ew}(f_{\lambda})$ vanishes, then (recall
\eqref{M1lambda}),
$\E_{\nu_{\rho}}[\tilde{M}_1(\lambda,l)^2]=D_{ew}(f_{\lambda})\rightarrow
0$ and the limit variance is just $\E_{\nu_{\rho}}\left[(M_1\cdot
l)^2\right]>0.$

It remains to show the claim in \eqref{Dirichlet2}. Indeed, for an
arbitrary $f$, we can estimate

\begin{equation}\label{est1}
\begin{aligned}
|\langle \phi\cdot l,f\rangle_{\nu_\rho}|&=\left|\frac{1}{2}\int d\nu_\rho
\sum_y (y\cdot l)
\left[c_{0,y}(\eta)-c_{0,-y}(\eta)\right]f(\eta)\right|\\
&=\left|\frac{1}{2} \sum_{|y|_1\leq R} \int d\nu_\rho (y\cdot l) \eta(0)
\left[\eta(y)-\eta(-y)\right]f(\eta)\right|\\
&\leq \frac{1}{2}\sum_{|y|_1\leq R} |y\cdot l|\left|\int d\nu_\rho
\left[\eta(y)-\eta(-y)\right]f(\eta)\right|.
\end{aligned}
\end{equation}

Note that due to the irreducibility of $p(\cdot)$, for any
$y\in\Z^d$ with $|y|_1\leq R$, we can write
$$\eta(y)-\eta(-y)=\sum_{i=1}^n [\eta(z_i)-\eta(z_{i-1})] $$
for some sequence $(z_0=y,z_1,\ldots,z_n=-y)$, with
$p(z_{i}-z_{i-1})>0$ for $i=1,\ldots,n$.
Moreover

\begin{equation}\label{est2}
\begin{aligned}
&\left|\int d\nu_\rho \left[\eta(z_i)-\eta(z_{i-1})\right]f(\eta)\right|=
\left|\int d\nu_\rho \,\eta(z_{i-1})\left[f(\eta^{z_{i-1},z_i})-f(\eta)\right]\right|\\
&\leq \rho\left(\int d\nu_\rho \left[f(\eta^{z_{i-1},z_i})-f(\eta)\right]^2\right)^{1/2}
 \leq p(z_i-z_{i-1})^{-1/2} D_{se}(f)^{1/2}.
\end{aligned}
\end{equation}
Combining \eqref{est1} and \eqref{est2}, we obtain
\eqref{Dirichlet2} which concludes the proof.
\end{proof}

\subsection{Concluding remarks}

\begin{remark}{\bf{(On the tagged particle in symmetric exclusion)
}}\newline In the original paper by Kipnis-Varhadan, the authors
used their general theorem to show the diffusivity of a tagged
particle in the symmetric exclusion process in any dimension. An
exceptional case is when the symmetric exclusion is nearest-neighbor
and one-dimensional, which has been shown to be sub-diffusive \cite{Ar83}
due to the ``traffic jam'' created by the other particles in the
system. In particular, in this latter context, the analogous two
martingales involved in \eqref{RWrepresentation2} do annihilate
each other. In fact, the crucial estimate in \eqref{Dirichlet2}
does not hold.
\end{remark}

\begin{remark}{\bf{(Particle systems as non-elliptic dynamical
random conductances)}}\newline  The model we introduced is an
example of time-dependent random conductances, non-elliptic from
below, but bounded from above since $c_{\{x,y\}}(t)\in\{0,1\}$. In
a similar fashion, we can interpret more general particle systems
as models of non-elliptic dynamical random conductances, even
unbounded from above. This can be done by considering a particle
system $\xi_t\in\N^{\Z^d}$ and again setting
$c_{x,y}(t)=\xi_t(x)\xi_t(y)$ (e.g. a Poissoinian field of
independent random walks), provided that the particle system has
``well behaving" space-time correlations and good spectral
properties. Furthermore, in principle, Theorem
\ref{Functional CLT} can be pushed to obtain the analogous
quenched statement. We plan to address these natural
generalizations in future work.
\end{remark}

\end{document}